\documentclass[graybox]{svmult}

\usepackage{type1cm}        
\usepackage{makeidx}         
\usepackage{graphicx}        
\usepackage{multicol}        
\usepackage[bottom]{footmisc}

\usepackage{newtxtext}       %
\usepackage{newtxmath}       

\makeindex             

\usepackage{amsmath,amsfonts}
\usepackage{graphicx}

\usepackage{algorithm}
\usepackage{algpseudocode}

\newcommand{\bit}{\begin{itemize}}
\newcommand{\eit}{\end{itemize}}

\newcommand{\Dam}{\mathcal{D}}
\newcommand{\spe}{\psi}
\newcommand{\bzero}{\bold{0}}
\newcommand{\x}{\mathbf{x}}
\newcommand{\y}{\mathbf{y}}
\renewcommand{\u}{\mathbf{u}}
\renewcommand{\E}{\mathbf{E}}
\newcommand{\A}{\mathbf{A}}
\newcommand{\B}{\mathbf{B}}
\renewcommand{\b}{\mathbf{b}}
\newcommand{\C}{\mathbf{C}}
\newcommand{\N}{\mathbf{N}}

\renewcommand{\H}{\mathbf{H}}
\newcommand{\V}{\mathbf{V}}
\newcommand{\W}{\mathbf{W}}
\newcommand{\X}{\mathbf{X}}
\renewcommand{\I}{\mathbf{I}}
\renewcommand{\L}{\mathbf{L}}

\renewcommand{\S}{\mathbf{S}}
\renewcommand{\P}{\mathbf{P}}
\newcommand{\Q}{\mathbf{Q}}

\begin{document}

\title*{Balanced Truncation Model Reduction \\ for Lifted Nonlinear Systems}
\author{Boris Kramer and Karen Willcox}
\institute{Boris Kramer \at Department of Mechanical and Aerospace Engineering, University of California San Diego, \email{bmkramer@ucsd.edu}
\and Karen Willcox \at Oden Institute for Computational Engineering Science, University of Texas at Austin, \email{kwillcox@oden.utexas.edu}}
%
%
\maketitle

\abstract{
We present a balanced truncation model reduction approach for a class of nonlinear systems with time-varying and uncertain inputs. 
First, our approach brings the nonlinear system into quadratic-bilinear~(QB) form via a process called lifting, which introduces transformations via auxiliary variables to achieve the specified model form. 
Second, we extend a recently developed QB balanced truncation method to be applicable to such lifted QB systems that share the common feature of having a system matrix with zero eigenvalues. 
We illustrate this framework and the multi-stage lifting transformation on a tubular reactor model. In the numerical results we show that our proposed approach can obtain reduced-order models that are more accurate than proper orthogonal decomposition reduced-order models in situations where the latter are sensitive to the choice of training data. 	
}

\section{Introduction}
We consider reduced-order modeling for large-scale nonlinear systems with time-dependent and uncertain inputs, as such models appear in many practical engineering applications. A reduced-order model (ROM) has to capture the rich dynamics resulting from time-dependent inputs~\cite{antoulas05}. Proper orthogonal decomposition (POD)~\cite{holmes_lumley_berkooz_1996} is the most common model reduction framework for nonlinear systems. As a trajectory-based method, POD relies on user-specified training data, leaving the method vulnerable to poor choices in training inputs.  

Interpolatory model reduction approaches that focus on the transfer function mapping from inputs to outputs---and do not require training data---are mature for linear systems~\cite{antoulas05,antoulas2010interpolatory}. Moreover, significant progress has been made for input-output model reduction for nonlinear systems with known governing equations~\cite{bennerBreiten2015twoSided,bennergoyal2016QBIRKA,cao2018krylovQB}. In the situation where the governing equations are unknown and only frequency-domain input-output data is available, Antoulas and Gosea~\cite{gosea2018LoewnerQB} proposed a data-driven reduced-order modeling framework for quadratic-bilinear~(QB) systems. 

Balanced truncation methods offer another promising avenue for model reduction. Initially proposed by Moore~\cite{moore81principal} for linear systems, many extensions are available, such as snapshot-based approximations (in the stable~\cite{willcox2002balanced} and unstable case~\cite{flinois2015BTunstableLTI}), as well as weighted and time-limited balanced truncation methods, see e.g., the survey~\cite{gugercin2004survey}. 
Extending balancing transformations to large-scale nonlinear systems is an open problem, since the balancing transformations become state-dependent (see~\cite{scherpen1993balancing,fujimoto2010balanced}) and are hence impractical when the model dimension is large. To overcome this computational bottleneck, approximate Gramians that are state-independent (as they are in the linear case) can be used.
Balanced truncation for nonlinear systems based on algebraic Gramians~\cite{verriest2006algebraicGramiansNLBT} requires the solution of Lyapunov-type equations, and is hence appealing from a computational perspective. A computationally efficient framework via truncated Gramians for QB systems has been proposed in~\cite{bennerGoyal2017BT_quadBilinear}.
Empirical Gramians for nonlinear systems~\cite{lall2002subspace} can also be used, yet their computation requires as many simulations of the full systems as there are inputs and outputs: each simulation uses an impulse disturbance per input channel while setting the other inputs to zero. 
A method for approximating the nonlinear balanced truncation reduction map  via reproducing kernel Hilbert spaces has been proposed in~\cite{bouvrie2017kernel}. 

In this work, we propose a balanced truncation model reduction method for a class of nonlinear systems with time-varying inputs. As a first step, we use variable transformations and \textit{lifting}, which allows us to bring the nonlinear system in QB form via the introduction of auxiliary variables. Lifting transformations that transform general nonlinear systems into polynomial models have been explored over several decades in different communities~\cite{mccormick1976computability, kerner1981universal,savageau1987recasting,gu2011qlmor,bennerBreiten2015twoSided,liu2015abstraction,KW18nonlinearMORliftingPOD,SKHW2020_learning_ROMs_combustor}.  
We show that the lifted QB models often have a special structure, namely that the linear system matrix becomes has zero eigenvalues, making the balancing algorithm for QB systems~\cite{bennerGoyal2017BT_quadBilinear} not directly applicable due to the typical requirement that the system matrix be stable. We thus propose a modified approach for balanced truncation of lifted QB systems, which uses the definition of the lifted variable to artificially introduce a stabilization to the system.
As an example, we consider a  nonlinear tubular reactor model for which we present a lifting transformation that brings the model into QB form, and subsequently perform balancing.

This paper is organized as follows. Section~\ref{sec:truncGramians} reviews truncated Gramians for QB systems. Section~\ref{sec:Lifting_structure} illustrates that lifted QB models often have a common structure and also introduces the tubular reactor model and its lifted QB model. Section~\ref{sec:balancingLIftedSys} proposes our balancing approach for lifted QB systems with special structure. Section~\ref{sec:numerics} presents numerical results for the tubular reactor model.

\section{Quadratic-bilinear systems and balancing} \label{sec:truncGramians}
A QB system of ordinary differential equations can be written as	
\begin{align}
	\dot{\x} & = \A \x + \H (\x \otimes \x) + \sum_{k=1}^m (\N_k \x) \u_k + \B \u, \label{eq:QB1}\\
	\y &= \C\x,\label{eq:QB2}
\end{align}
where $\x(t) \in \mathbb{R}^{n}$ is the state, $t\geq 0$ denotes time, the initial condition is $\x(0)=\x_0$,  $\u(t) \in \mathbb{R}^m$ is a time-dependent input and $\u_k$ its $k$th component, $\B\in \mathbb{R}^{n\times m}$ is the input matrix, $\A \in \mathbb{R}^{n\times n}$ is the system matrix, $\C\in \mathbb{R}^{p\times n}$ is the output matrix and the $\N_k \in \mathbb{R}^{n\times n}, k=1, \ldots, m$ represent the bilinear coupling between input and state. 
The matrix $\H$ is viewed as a mode-1 matricization of a tensor $\mathcal{H}\in \mathbb{R}^{n\times n \times n}$, so $\mathcal{H}^{(1)} = \H$. Moreover, $\mathcal{H}^{(2)}$ denotes the mode-2 matricization of the tensor $\mathcal{H}$. We assume without loss of generality (see~\cite[Sec 3.1]{bennerBreiten2015twoSided}) that the matrix $\H$ is symmetric in that $\H(\x_1 \otimes \x_2) = \H(\x_2 \otimes \x_1)$. 

For system~\eqref{eq:QB1}--\eqref{eq:QB2} the controllability and observability energy functions are defined as
\begin{align*}
	\mathcal{E}_c(\x_0) & := \min_{\u\in L_2(-\infty,0), \x(-\infty)=\bzero, \x(0)=\x_0} \ \frac{1}{2} \int_{-\infty}^{0} \Vert \u(t) \Vert^2 \text{d}t, \\
	\mathcal{E}_o(\x_0) & :=\frac{1}{2} \int_{0}^{\infty} \Vert \y(t) \Vert^2 \text{d}t, 
\end{align*}
where $\mathcal{E}_c$ quantifies the minimum amount of energy required to steer the system from $\x(-\infty)=0$ to $\x(0)=\x_0$, and $\mathcal{E}_o$ quantifies the output energy generated by the nonzero initial condition $\x_0$ and $\u(t)\equiv0$.

Balanced truncation achieves reduction in the dimension of the state space by eliminating those states that are hard to control (large $\mathcal{E}_c$) and hard to observe (small $\mathcal{E}_o$). Thus, an important component of balanced truncation model reduction is to find the coordinate transformation that ranks the controllability and observability of the states. The \textit{Gramian} matrices (formally introduced below) are central to finding this balanced coordinate transformation for system~\eqref{eq:QB1}--\eqref{eq:QB2}. In particular, if symmetric positive definite Gramians $\P = \L_\P {\L_\P}^\top$ and $\Q=\L_\Q {\L_\Q}^\top$  for the QB system~\eqref{eq:QB1}--\eqref{eq:QB2} are available, we can obtain transformation matrices $\V$ and $\W$, with $\V = \W^{-1}$ that diagonalize the Gramians, $\V^\top \P \V = \W^\top \Q \W =\mathcal{S}=\text{diag}{(\sigma_1, \ldots, \sigma_n)}$, where $\sigma_i$ are the singular values\footnote{In the linear case those would be the Hankel singular values.} of ${\L_\Q}^\top \L_\P$.   The energy functions in the balanced state $\tilde{\x} = \W \x$ can then be bounded as $ \mathcal{E}_c(\tilde{\x})\geq \frac{1}{2}\tilde{\x}^\top \mathcal{S}^{-1} \tilde{\x}$, and $ \mathcal{E}_o(\tilde{\x})\leq \frac{1}{2} \tilde{\x}^\top \mathcal{S} \tilde{\x}$ in a neighborhood of the origin \cite[Sec.4]{bennerGoyal2017BT_quadBilinear}.
From these inequalities, we see that states corresponding to small singular values of ${\L_\Q}^\top \L_\P$ are \textit{hard to control} and \textit{hard to observe}. The subspaces spanned by the singular vectors corresponding to small singular values  can thus be discarded in a ROM. 

We now turn to the computation of the Gramian matrices $\P$ and $\Q$. For the special case of QB systems, the next proposition states conditions under which approximate algebraic Gramians exist, and suggests a computational framework to find those. 

\begin{proposition} [{\cite[Cor.~3.4]{bennerGoyal2017BT_quadBilinear}}]
	Consider the QB system~\eqref{eq:QB1}--\eqref{eq:QB2}. If $\A$ is a stable matrix, then the \textit{truncated Gramians} $\P_{\mathcal{T}}, \Q_{\mathcal{T}}$ are defined as solutions to 
	\begin{align}
	\A\P_{\mathcal{T}} + \P_{\mathcal{T}}\A^\top + \H(\P_1\otimes \P_1) \H^\top + \sum_{k=1}^m \N_k \P_1 \N_k^\top + \B\B^\top & = \bzero \label{eq:genLyapP}, \\
	\A^\top \Q_{\mathcal{T}} + \Q_{\mathcal{T}} \A + \mathcal{H}^{(2)}(\P_1\otimes \Q_1)(\mathcal{H}^{(2)})^\top + \sum_{k=1}^m \N_k^\top \Q_1 \N_k + \C^\top \C & =\bzero, \label{eq:genLyapQ}
	\end{align}
	where $\P_1$ and $\Q_1$ are solutions to the standard linear Lyapunov equations
	\begin{equation}
	\A \P_1 + \P_1 \A^\top + \B\B^\top  = \bzero, \qquad 
	\A^\top \Q_1 + \Q_1 \A + \C^\top\C  = \bzero. \label{eq:LyapLin}
	\end{equation}
\end{proposition}
The Gramians $\P_{\mathcal{T}}, \ \Q_{\mathcal{T}}$ are used to obtain a QB reduced-order model (QB-ROM) as follows: Compute the singular value decomposition $\L_{\Q_\mathcal{T}}^\top \L_{\P_\mathcal{T}}= \mathcal{U} \mathcal{S} \mathcal{V}^\top$ and let its rank $r$ approximation be denoted as $\mathcal{U}_r \mathcal{S}_r \mathcal{V}_r^\top$. The projection matrices are $\W = \L_{\Q_\mathcal{T}} \mathcal{U}_r \mathcal{S}_r^{-1/2}$ and $\V = \L_{\P_\mathcal{T}} \mathcal{V}_r \mathcal{S}_r^{-1/2}$.
The balanced QB-ROM has the form 
\begin{align}
	\dot{\widehat{\x}} & = \widehat{\A} \widehat{\x} + \widehat{\H} (\widehat{\x} \otimes \widehat{\x}) + \sum_{k=1}^m \widehat{\N}_k \widehat{\x} \u_k + \widehat{\B} \u, \label{eq:ROMQB1}\\
	\widehat{\y} &= \widehat{\C}\widehat{\x}.   \label{eq:ROMQB2}
\end{align}
with the matrices $ \widehat{\A} = \W^\top \A \V, \quad \widehat{\B} = \W^\top \B, \quad \widehat{\C} = \C \V, \quad \widehat{\N}_k = \W^\top \N_k \V, \quad \widehat{\H} = \W^\top \H (\V\otimes \V)$. The Gramians of the QB-ROM are balanced, i.e., the matrices $\widehat{\P}_\mathcal{T}=\widehat{\Q}_\mathcal{T}$ are equal and diagonal.

\section{Lifting systems to QB form: Introducing structure} \label{sec:Lifting_structure}
In Section~\ref{sec:StructureQB} we illustrate that lifting transformations for nonlinear systems often lead to a specific structure of the lifted system.  Section~\ref{sec:tubReactor_model} introduces a nonlinear tubular reactor model, and its subsequent lifted QB model with this special structure.

\subsection{Structure in lifted QB systems} \label{sec:StructureQB}
As an example of a higher-order dynamical system that can be lifted to QB form, consider the nonlinear  ordinary differential equation
\begin{equation} \label{eq:FOMpoly}
	\dot{x} = \sum_{k=1}^d a_k x^k + bu, 
\end{equation}
where $x(t)$ is the one-dimensional state variable, $a_1, \ldots, a_d$ are known coefficients, $x^k$ is the $k$th power of $x(t)$, $u(t)$ is an input function,  and $b \in \mathbb{R}$ is a parameter. We consider here a one-dimensional ordinary differential equation (ODE) for illustration and note that the lifting framework directly applies to $n$-dimensional nonlinear systems. 
Our goal is to rewrite this system in QB form by introducing auxiliary variables. For this example, we define auxiliary variables $w_\ell := x^{\ell+1}$, so that the governing equation can be written as a linear system
\begin{equation} \label{eq:FOMpolyLifted}
\dot{x} = a_1 x + \sum_{k=1}^{d-1} a_{k-1} w_k + bu. 
\end{equation}
The auxiliary state dynamics are computed via the chain rule (or Lie derivative) as
$$
\dot{w}_\ell = \frac{d}{dt}x^{\ell+1} = (\ell+1) x^{\ell} \dot{x} = (\ell+1) w_{\ell-1} \left (  a_1 x + \sum_{k=1}^{d-1} a_{k-1} w_k + bu  \right ),
$$
and here $w_0:=x$.
Note that the right-hand-side  contains quadratic products $w_{\ell -1}(t)w_k(t)$ as well as bilinear products $b (\ell+1) u(t) w_{\ell -1}(t)$. Another choice to write the auxiliary dynamics would be to replace $w_{\ell-1}x$ with $w_\ell$. Next, we show two examples of polynomial and non-polynomial models where lifting leads to a similar structure.

\noindent \textbf{Example 1\ }
\textit{	Consider the ODE
	\begin{equation}
		\dot{x} = ax^3 + bu. \label{eq:ExampleODE}
	\end{equation}
	Our goal is to lift this system to QB form by introducing auxiliary variables. Let
	$$
	w_1=x^2,\quad  w_2 = x^3, 
	$$ 
	be the auxiliary variables. Note that the original dynamics then become linear, i.e., $\dot{x} = aw_2 + bu$. We compute the auxiliary state dynamics $\dot{w}_1 = 2 x \dot{x} = 2a x w_2 + 2xbu$ and $\dot{w}_2 = 3 x^2 \dot{x} = 3 w_1 (aw_2 + bu)$. 
	Taken together, the nonlinear equation~\eqref{eq:ExampleODE} with one state variable is equivalent to the QB-ODE with three state variables
	\begin{align} \label{eq:ex1-1}
		\dot{x}    =  aw_2  +  bu, \qquad 
		\dot{w}_1  =  2a xw_2 + 2bxu, \qquad 
		\dot{w}_2  =  3a w_1 w_2 + 3b w_1 u.
	\end{align}
	The system~\eqref{eq:ex1-1} is equivalent to the original nonlinear equation~\eqref{eq:ExampleODE}, in the sense that both systems yield the same solution $x(t)$.
	We can write \eqref{eq:ex1-1} in the QB form of equation~\eqref{eq:QB1} with the lifted state $\x=[x, w_1, w_2]^\top$ and matrices 
	\begin{equation} \label{eq:Ex1_matrices}
	\A = \left[
	\begin{array}{ccc}
	0 & 0 & a  \\
	0 & 0 & 0 \\
	0 & 0 & 0\\
	\end{array}
	\right], \
	\H = 
	\begin{bmatrix}
	0 & 0 & 0 & 0 & 0 & 0 & 0 & 0 & 0\\
	0 & 0 & 2a & 0 & 0 & 0 & 0 & 0 & 0\\	
	0 & 0 & 0 & 0 & 0 & 3a & 0 & 0 & 0\\
	\end{bmatrix}, \
	\N_1 = \left[
	\begin{array}{ccccc}
	0 & 0 & 0 \\
	2b & 0 & 0  \\	
	0 & 3b & 0 \\
	\end{array}
	\right], \
	\B = \left[
	\begin{array}{c}
	1 \\
	0 \\
	0 \\
	\end{array}
	\right].
	\end{equation} 
}

\noindent \textbf{Example 2\ }
\textit{Consider an ODE with non-polynomial term and linear term:
\begin{equation*}
\dot{x} = ax +e^{-x} + bu,
\end{equation*}
which we lift to QB form. We introduce $w_1 = e^{-x}$ as the auxiliary variable, so that $\dot{w}_1 = -w_1 ( ax + w_1 + bu ) = -aw_1x - w_1^2 -w_1bu$, which can be written in QB form as
\begin{equation} \label{eq:Ex2_matrices}
\begin{bmatrix} \dot{x} \\ \dot{w}_1 \end{bmatrix}
= 
\begin{bmatrix}
a & 1 \\
0 & 0 \\
\end{bmatrix}
\begin{bmatrix} x \\ w_1 \end{bmatrix}
+ 
\begin{bmatrix}
0 & 0 & 0 & 0\\
0 & 0 & -a & -1 \\
\end{bmatrix}
\left ( \begin{bmatrix} x \\ w_1 \end{bmatrix} \otimes \begin{bmatrix} x \\ w_1 \end{bmatrix} \right )
+ 
\begin{bmatrix}
0 & 0 \\
0 & -b \\
\end{bmatrix}
\begin{bmatrix} x \\ w_1 \end{bmatrix}
u 
+ \begin{bmatrix} b \\ 0 \end{bmatrix}
u.
\end{equation} 		
}

We observe that the lifted matrices in the previous two examples (equations~\eqref{eq:Ex1_matrices} and \eqref{eq:Ex2_matrices}) have the following block-structure:
\begin{equation}\label{eq:QB-structure}
\A   = \begin{bmatrix} \A_{11} & \A_{12} \\ \bzero & \bzero \end{bmatrix}, \quad 
\H  = \begin{bmatrix} \bzero  \\ \H_{2} \end{bmatrix},  \quad
\N  = \begin{bmatrix} \bzero & \bzero \\ \N_{21} & \N_{22} \end{bmatrix}, \quad
\B  = \begin{bmatrix} \B_{1} \\ \bzero  \end{bmatrix},
\end{equation}
where $\bzero$ is a matrix of zeros of appropriate dimensions. 
Note, that the system matrix of the lifted system~\eqref{eq:QB-structure} has zero eigenvalues. Moreover, the original input $\B u(t)$ only affects the original state, but appears as a bilinear term in the auxiliary states after application of the chain rule. Thus, the $\N$ matrix is nonzero only in the rows corresponding to the auxiliary states. The matrix structure in equation~\eqref{eq:QB-structure} commonly occurs in QB systems that were obtained from lifting transformations, for instance the systems given in Example~1 and Example~2. Nevertheless, due to the non-uniqueness of lifting transformations, the block structure in~\eqref{eq:QB-structure} is both a result of the original dynamical system form and the chosen lifting transformation.

\subsection{Tubular reactor model} \label{sec:tubReactor_model}
We consider a non-adiabatic tubular reactor model with single reaction as in~\cite{heinemann1981tubular} and follow the discretization in \cite{zhou2012thesis}. The model describes the evolution of the species concentration $\spe(s,t)$ and temperature $\theta(s,t)$ with spatial variable $s\in (0,1)$ and time $t >0$.  The PDE model is	
\begin{align}
	\frac{\partial \spe}{\partial t}   & = \frac{1}{Pe} \frac{\partial^2 \spe}{\partial s^2} - \frac{\partial \spe}{\partial s} - \Dam  f(\spe, \theta; \gamma), \label{eq:PDEa} \\
	\frac{\partial \theta}{\partial t} & = \frac{1}{Pe} \frac{\partial^2 \theta}{\partial s^2} - \frac{\partial \theta}{\partial s} - \beta (\theta - \theta_{\text{ref}}) + \mathcal{B}\Dam f(\spe,\theta; \gamma) + b u, \label{eq:PDEb}
\end{align}
with input function $u=u(t)$ and a function $b=b(s)$ encoding the influence of the input to the computational domain. The parameters are the Damk{\"o}hler number $\Dam$, P{\`e}clet number $Pe$ as well as known constants $\mathcal{B},\ \beta, \ \theta_{\text{ref}}, \gamma$. The polynomial nonlinear term that drives the reaction is
\begin{equation*}
	f(\spe, \theta; \gamma) = \spe (c_0 + c_1\theta + c_2\theta^2 + c_3\theta^3)
\end{equation*}
with given constants $c_0, \ldots c_3$.
%
\noindent Robin boundary conditions are imposed on the left boundary of the domain and Neumann boundary conditions on the right
\begin{align*}
	\frac{\partial \spe}{\partial s}(0,t) & = Pe (\spe(0,t)-1),    &\frac{\partial \theta}{\partial s}(0,t) & = Pe (\theta(0,t)-1), \\
	\frac{\partial \spe}{\partial s}(1,t) &= 0, 					 &\frac{\partial \theta}{\partial s}(1,t) &=0.
\end{align*}
The initial conditions are prescribed as $\spe(s,0) = \spe_0(s)$, and  $\theta(s,0) = \theta_0(s)$.
The output of interest is the temperature oscillation at the reactor exit, i.e., the quantity $y(t) = \theta(s=1,t)$.
The diffusive and convective terms are approximated with second-order centered differences in the interior of the domain. Second-order forward and backward difference schemes are used for the inflow and outflow boundary conditions, respectively.
The finite difference approximation of dimension $2n$ of the tubular reactor PDE  is then
\begin{align}
	\dot{\boldsymbol{\spe}} & = \A_\spe \boldsymbol{\spe} + \b_\spe - \mathcal{D} \ \boldsymbol{\spe} \odot  (\boldsymbol{c}_0 + \boldsymbol{c}_1\odot \boldsymbol{\theta} + \boldsymbol{c}_2\odot \boldsymbol{\theta}^2 + \boldsymbol{c}_3\odot \boldsymbol{\theta}^3), \label{eq:tub_FOM1}\\
	\dot{\boldsymbol{\theta}} & = \A_\theta \boldsymbol{\theta} + \b_\theta + \b u+ \mathcal{B}\mathcal{D} \ \boldsymbol{\spe} \odot  (\boldsymbol{c}_0 + \boldsymbol{c}_1\odot\boldsymbol{\theta} + \boldsymbol{c}_2\odot \boldsymbol{\theta}^2 + \boldsymbol{c}_3\odot\boldsymbol{\theta}^3) \label{eq:tub_FOM2}.
\end{align}
Here, the (Hadamard) componentwise product of two vectors is denoted as $[\boldsymbol{\spe} \odot\boldsymbol{\theta}]_i = \boldsymbol{\spe}_i\boldsymbol{\theta}_i$, and the powers of vectors are also taken componentwise. The constant vector $ \b_\spe$ encodes the boundary condition, and  $\b_\theta$ is the sum of contributions from the boundary conditions and the $\beta\cdot\theta_\text{ref}$ term in equation~\eqref{eq:PDEb}. The term $\A_\spe \boldsymbol{\spe}$ is a discretized version of the advection-diffusion terms $\frac{1}{Pe} \spe_{ss} - \spe_s $ and $\A_\theta \boldsymbol{\theta}$ is a discretized representation of the $\frac{1}{Pe} \theta_{ss} - \theta_s - \beta \theta$ terms. 

We quadratic-bilinearize the finite dimensional system via a lifting transformation.\footnote{Lifting is not unique, and there might be lower-order quadratic systems than the one suggested here. However, how to obtain them is an open problem.}\footnote{Lifting for a tubular reactor model with Arrhenius reaction term has been considered by the authors in \cite{KW18nonlinearMORliftingPOD}, however that lifting transformation resulted in QB-DAEs, for which balancing model reduction is an open problem.}
To lift the system, we introduce the auxiliary variables
\begin{equation*}
	\mathbf{w}_1 = \boldsymbol{\spe}\odot\boldsymbol{\theta}, \quad \mathbf{w}_2 = \boldsymbol{\spe}\odot\boldsymbol{\theta}^2, \quad \mathbf{w}_3 = \boldsymbol{\spe} \odot \boldsymbol{\theta}^3,\quad  \mathbf{w}_4 = \boldsymbol{\theta}^2, \quad \mathbf{w}_5 = \boldsymbol{\theta}^3,
\end{equation*}
so that we have the equivalent QB system of dimension $7n$:
\begin{align}
	\dot{\boldsymbol{\spe}} & = [\A_\spe - \mathcal{D}\text{diag}(\boldsymbol{c}_0)] \boldsymbol{\spe} + \mathbf{b}_\spe - \mathcal{D}  (\boldsymbol{c}_1 \odot\mathbf{w}_1 + \boldsymbol{c}_2 \odot\mathbf{w}_2 + \boldsymbol{c}_3\odot \mathbf{w}_3), \ \label{eq:spe}  \\
	\dot{\boldsymbol{\theta}} & = \A_\theta \boldsymbol{\theta} + \mathcal{B}\mathcal{D}\boldsymbol{c}_0\odot \boldsymbol{\psi} + \mathbf{b}_\theta + \b u + \mathcal{B}\mathcal{D} \ (\boldsymbol{c}_1 \odot\mathbf{w}_1 + \boldsymbol{c}_2\odot \mathbf{w}_2 + \boldsymbol{c}_3\odot \mathbf{w}_3), \label{eq:temp} \\
	\dot{\mathbf{w}}_1 & = \dot{\boldsymbol{\spe}}\odot\boldsymbol{\theta} + \boldsymbol{\spe}\odot\dot{\boldsymbol{\theta}}, \label{eq:w1}\\
	\dot{\mathbf{w}}_2 & = \dot{\boldsymbol{\spe}}\odot\mathbf{w}_4 + 2\mathbf{w}_1\odot\dot{\boldsymbol{\theta}}, \\
	\dot{\mathbf{w}}_3 & = \dot{\boldsymbol{\spe}}\odot\mathbf{w}_5 + 3\mathbf{w}_2\odot\dot{\boldsymbol{\theta}}, \\
	\dot{\mathbf{w}}_4 & = 2\boldsymbol{\theta} \odot\dot{\boldsymbol{\theta}}, \\
	\dot{\mathbf{w}}_5 & = 3\mathbf{w}_4 \odot\dot{\boldsymbol{\theta}}. \label{eq:w5}
\end{align}

Note, that the original state equations for species and temperature, \eqref{eq:spe}--\eqref{eq:temp}, are linear in the new state variables, whereas the differential equations for the added state variables \eqref{eq:w1}--\eqref{eq:w5} are quadratic in the state variables (after inserting $\dot{\boldsymbol{\spe}}, \dot{\boldsymbol{\theta}}$, which are linear). Hence, the system is of QB form \eqref{eq:QB1}--\eqref{eq:QB2},
where $\x(t) = [\boldsymbol{\spe}^\top \ \boldsymbol{\theta}^\top \ \mathbf{w}_1^\top \ \mathbf{w}_2^\top \ \mathbf{w}_3^\top \ \mathbf{w}_4^\top \ \mathbf{w}_5^\top]^\top \in \mathbb{R}^{7n}$ is the new lifted state.
The lifted state equations \eqref{eq:spe}--\eqref{eq:w5} again induce a specific structure of the system matrices:
\begin{subequations} \label{eq:QBODE_tubular_matrices}
	\begin{align}
		\A  & = \begin{bmatrix} \A_{11} & \A_{12} \\ \bzero_{5n\times 2n} & \bzero_{5n}\end{bmatrix}  \in \mathbb{R}^{7n\times 7n}, \quad \A_{11} = \begin{bmatrix} \A_\spe - \mathcal{D}\text{diag}(\boldsymbol{c}_0) & \bzero \\ \bzero & \A_\theta  \end{bmatrix}, \\
		\H & = \begin{bmatrix} \bzero_{2n\times 4n^2} & \bzero_{2n\times 45n^2} \\ \H_{21} & \H_{22} \end{bmatrix} \in \mathbb{R}^{7n\times (7n)^2}, \qquad
		\N  = \begin{bmatrix} \bzero_{2n\times 2n} & \bzero_{2n \times 5n} \\ \N_{21} & \N_{22} \end{bmatrix} \in \mathbb{R}^{7n\times 7n}, \\
		\B & = \begin{bmatrix} \b_\spe & \bzero \\ \b_\theta & \b \\ \bzero_{5n} & \bzero_{5n} \end{bmatrix} \in \mathbb{R}^{7n\times 2}, \qquad
		\C  = \begin{bmatrix} \C_1 \ \  \bzero_{5n} \end{bmatrix} \in \mathbb{R}^{1 \times 7n}, \quad \C_1 = [\bzero_{2n-1} \ 1].
	\end{align}
\end{subequations}
Note that the original dynamics are linear after the lifting is applied, and the auxiliary states have no linear parts, as the time derivative of the auxiliary variables consistently results in purely quadratic dynamics. Thus, the matrix $\A$ has zero eigenvalues, and $\H$ is such that there is no contribution in the rows corresponding to the original state. Moreover, the bilinear state-input interactions only occur in the auxiliary states. This structure is exactly the one presented in equation~\eqref{eq:QB-structure}.

\section{Balancing for lifted systems with special structure} \label{sec:balancingLIftedSys}
Section~\ref{subsec:stabilization} introduces an artificial numerical parameter that guarantees existence of truncated Gramians. Section~\ref{subsec:BalLin} derives the Gramians for the linear part of the system, while Section~\ref{subsec:balQuadratic} presents the truncated Gramians for the QB system.

\subsection{Artificial stabilization} \label{subsec:stabilization}
The special structure of the system matrix $\A$ that arises from lifting transformations leads to zero eigenvalues. Hence, there are no unique solutions to the Lyapunov equations~\eqref{eq:genLyapP}--\eqref{eq:LyapLin} and the standard QB balanced truncation method from \cite{bennerGoyal2017BT_quadBilinear} cannot be applied. To circumvent this problem, we observe that if the original nonlinear system is polynomial (cf. \eqref{eq:FOMpoly} and its lifted version~\eqref{eq:FOMpolyLifted}) then $-\alpha w_\ell + \alpha x^{\ell+1} =0$ for any $\alpha\in \mathbb{R}$. We can use this equality to artificially stabilize the linear term. For example, in equation~\eqref{eq:w1} of the tubular reactor ODE, we can introduce a term $-\alpha \mathbf{w}_1 + \alpha (\boldsymbol{\spe}\odot \boldsymbol{\theta})$. Applying this to every auxiliary variable, the linear system matrix becomes  
\begin{equation}
	\A(\alpha) = \begin{bmatrix} \A_{11} & \A_{12} \\ \bzero & -\alpha \I \end{bmatrix}, \label{eq:Aeps}
\end{equation}
and the quadratic tensor $\H(\alpha) = \H + \widetilde{\H}(\alpha)$, where $\widetilde{\H}(\alpha)$ contains the artificially added terms such that $-\alpha [\bzero^\top, \mathbf{w}^\top]^\top + \widetilde{\H}(\alpha)(\x \otimes \x) = \bzero$. With this system matrix $\A(\alpha)$, the Lyapunov equations \eqref{eq:genLyapP}--\eqref{eq:LyapLin} have unique positive semi-definite solutions if $\A_{11}$ is stable. Similar concepts have been used recently in~\cite{pulch2017balanced}. For all results in this section to hold, we require that $\A_{11}$ is stable.

\subsection{Balancing with Gramians of linearized system} \label{subsec:BalLin}
The Lyapunov equations for the linear system part, Eq.~\eqref{eq:LyapLin} are needed to compute the truncated Gramians for the QB system. The next two propositions detail the computation of those linear Lyapunov equations. 

\begin{proposition} \label{prop:LinLyap}
	The solutions to the standard linear Lyapunov equations \eqref{eq:LyapLin} with $\A(\alpha)$ from equation~\eqref{eq:Aeps} are given by 
	\begin{equation} \label{eq:P1Q1}
		\P_1 = \begin{bmatrix} \P_{11} & \bzero \\ \bzero & \bzero  \end{bmatrix},
		\qquad 
		\Q_1 = \begin{bmatrix} \Q_{11} & \Q_{12}(\alpha) \\ \Q_{12}^\top(\alpha) & \frac{1}{2\alpha} \widetilde{\Q}_{22}(\alpha) \end{bmatrix} ,
	\end{equation}	
	where 
		\begin{align*}
		\bzero & = 	\A_{11} \P_{11} + \P_{11} \A_{11}^\top + \B_1 \B_1^\top, \\
			\bzero &  = \A_{11}^\top \Q_{11} + \Q_{11} \A_{11} + \C_1^\top \C_1, \\
			\Q_{12}(\alpha)&  = - (\A_{11}^\top - \alpha \I)^{-1} \Q_{11} \A_{12}, \\
			\widetilde{\Q}_{22}(\alpha)  & = \A_{12}^\top \Q_{12}(\alpha) + \Q_{12}(\alpha)^\top \A_{12}.
		\end{align*}
\end{proposition}
\begin{proof}
	For the controllability Gramian computation we have that
	\begin{equation*}
		\bzero = 	\begin{bmatrix} \A_{11} & \A_{12} \\ \bzero & -\alpha \I \end{bmatrix} 
		\begin{bmatrix} \P_{11} & \P_{12} \\ \P_{12}^\top & \P_{22} \end{bmatrix} + 
		\begin{bmatrix} \P_{11} & \P_{12} \\ \P_{12}^\top & \P_{22} \end{bmatrix}
		\begin{bmatrix} \A_{11}^\top & \bzero \\  \A_{12}^\top  & -\alpha \I \end{bmatrix}
		+ \begin{bmatrix} \B_{1}\B_1^\top & \bzero \\ \bzero & \bzero \end{bmatrix}, 
	\end{equation*}
	which results in the four equations
	\begin{align*}
		\A_{11} \P_{11} + \A_{12} \P_{12}^\top + \P_{11} \A_{11}^\top +\P_{12} \A_{12}^\top +\B_{1}\B_1^\top & =\bzero,  \\
		(\A_{11} - \alpha \I ) \P_{12} +\A_{12} \P_{22}&  = \bzero, \\
		\P_{12}^\top (\A_{11}^\top - \alpha \I ) + \P_{22} \A_{12}^\top &= \bzero, \\
		-2 \alpha \P_{22} &= \bzero.
	\end{align*}	
	Since $(\A_{11}^\top - \alpha \I )$ is invertible, the last three equations yield that $\P_{22}=\bzero$ and $\P_{12}=\bzero$. 	
	Next, we consider the observability Lyapunov equation
	\begin{equation*}
		\bzero = \begin{bmatrix} \A_{11}^\top & \bzero \\  \A_{12}^\top  & -\alpha \I \end{bmatrix}
		\begin{bmatrix} \Q_{11} & \Q_{12} \\ \Q_{12}^\top & \Q_{22} \end{bmatrix} + 
		\begin{bmatrix} \Q_{11} & \Q_{12} \\ \Q_{12}^\top & \Q_{22} \end{bmatrix}
		\begin{bmatrix} \A_{11} & \A_{12} \\ \bzero & -\alpha \I \end{bmatrix} 
		+ \begin{bmatrix} \C_1^\top \C_1 & \bzero \\ \bzero & \bzero \end{bmatrix} .
	\end{equation*}
	Blockwise multiplication results in the three distinct equations
	\begin{align*}
		\A_{11}^\top \Q_{11} + \Q_{11} \A_{11} + \C_1^\top \C_1  &  = \bzero,\\
		(\A_{11}^\top - \alpha \I) \Q_{12}  +  \Q_{11} \A_{12} & =\bzero, \\
		\A_{12}^\top \Q_{12} + \Q_{12}^\top \A_{12} - 2\alpha \Q_{22} & = \bzero,
	\end{align*}
	which we can solve successively to get $\Q_{11}, \Q_{12}, \Q_{22}$. We define $\widetilde{\Q}_{22}{(\alpha)} = \A_{12}^\top \Q_{12}{ (\alpha)} + \Q_{12}^\top{(\alpha)} \A_{12}$ so that the explicit dependence on $1/\alpha$ appears as $\Q_{22}(\alpha) = \frac{1}{2\alpha} \widetilde{\Q}_{22}(\alpha)$.
\end{proof}	
\begin{proposition}
	Consider the Gramians from Proposition~\ref{prop:LinLyap}, and let $\P_{11} = \L_{\P_{11}} \L_{\P_{11}}^\top$ and $\Q_{11} = \L_{\Q_{11}} \L_{\Q_{11}}^\top$ be Cholesky factorizations. Let the singular value decomposition approximate the product $\L_{\Q_{11}}^\top \L_{\P_{11}} \approx \mathcal{U}_{r} \mathcal{S}_r \mathcal{V}_r^\top$. Then the projection matrices for the balancing transformation are given as 
	\begin{equation} \label{eq:VWlinear}
		\V = \begin{bmatrix} \L_{\P_{11}} \mathcal{V}_r \mathcal{S}_r^{-1/2} \\ \bzero  \end{bmatrix}, \qquad 
		\W = \begin{bmatrix} \L_{\Q_{11}} \mathcal{U}_{r}  \mathcal{S}_r^{-1/2}   \\   \Q_{12}(\alpha)^\top \L_{\Q_{11}}^{-\top} \mathcal{U}_r \mathcal{S}_r^{-1/2}    \end{bmatrix}.
	\end{equation}
\end{proposition}
\begin{proof}
	To obtain the balancing transformation, we need to compute the product $\L_{\Q_1}^\top \L_{\P_1}$. We have that 
	\begin{equation*}
		\P_1 = \L_{\P_1} \L_{\P_1}^\top = \begin{bmatrix} \L_{\P_{11}} & \bzero \\ \bzero & \bzero   \end{bmatrix} \begin{bmatrix} \L_{\P_{11}}^\top & \bzero \\ \bzero & \bzero   \end{bmatrix}. 
	\end{equation*}
	For the Cholesky factorization of $\Q_1$ we introduce the Schur complement $ \S(\alpha) = \frac{1}{2\alpha}\widetilde{\Q}_{22}(\alpha) - \Q_{12}(\alpha)^\top \Q_{11}^{-1} \Q_{12}(\alpha) = \L_{\S(\alpha)} \L_{\S(\alpha)}^\top$,  so that we have 
	\begin{equation*}
		\Q_1 = \L_{\Q_1} \L_{\Q_1}^\top = \begin{bmatrix} \L_{\Q_{11}} & \ \ \bzero \\ \Q_{12}(\alpha)^\top \L_{\Q_{11}}^{-\top} & \ \ \L_{\S(\alpha)}  \end{bmatrix}   \begin{bmatrix} \L_{\Q_{11}}^\top & \ \  \L_{\Q_{11}}^{-1} \Q_{12}(\alpha)  \\  \bzero & \ \  \L_{\S(\alpha)}^\top  \end{bmatrix},
	\end{equation*}
	as follows from the results for the block-diagonal form of the Cholesky factorization (which can be verified by direct computation). 
	To obtain the balancing transformation, we consider the approximation via singular value decomposition 
	\begin{equation*}
		\L_{\Q_1}^\top \L_{\P_1} = \begin{bmatrix} \L_{\Q_{11}}^\top & \  \L_{\Q_{11}}^{-1} \Q_{12}(\alpha)  \\ \bzero & \ \L_{\S(\alpha)}^\top  \end{bmatrix} \begin{bmatrix} \L_{\P_{11}} & \ \bzero \\ \bzero & \ \bzero   \end{bmatrix}
		= \begin{bmatrix}  \L_{\Q_{11}}^\top \L_{\P_{11}} & \bzero \\ \bzero & \bzero   \end{bmatrix} \approx  \begin{bmatrix} \mathcal{U}_r \\ \bzero  \end{bmatrix} \mathcal{S}_r [ \mathcal{V}_r^\top \  \bzero ].
	\end{equation*}
	We then compute
	\begin{align*}
		\V & = \begin{bmatrix} \L_{\P_{11}} & \bzero \\ \bzero & \bzero   \end{bmatrix} \begin{bmatrix} \mathcal{V}_r \\ \bzero  \end{bmatrix} \mathcal{S}_r^{-1/2} = \begin{bmatrix} \L_{\P_{11}} \mathcal{V}_r \mathcal{S}_r^{-1/2} \\ \bzero  \end{bmatrix}, \\
		\W & = \begin{bmatrix} \L_{\Q_{11}} & \ \ \bzero \\ \Q_{12}(\alpha)^\top \L_{\Q_{11}}^{-\top} &\ \ \L_{\S(\alpha)}  \end{bmatrix}  \begin{bmatrix} \mathcal{U}_r \\ \bzero \end{bmatrix}  \mathcal{S}_r^{-1/2} 
		= \begin{bmatrix}  \L_{\Q_{11}} \mathcal{U}_r  \mathcal{S}_r^{-1/2} \\  \Q_{12}(\alpha)^\top \L_{\Q_{11}}^{-\top} \mathcal{U}_r \mathcal{S}_r^{-1/2} \end{bmatrix}.
	\end{align*}
	which completes the proof. 
\end{proof}

The two propositions show that only Lyapunov equations in the original model dimensions are necessary (and not in the lifted dimensions), which is computationally appealing. For example, for the tubular reactor model in Section~\ref{sec:tubReactor_model}, the original model dimensions are $2n$ and the lifted dimensions are $7n$ and we only solve $2n$-dimensional Lyapunov equations.

\subsection{Balancing with truncated (quadratic) Gramians } \label{subsec:balQuadratic}
The next two propositions consider the solution of the  Lyapunov equations~\eqref{eq:genLyapP}--\eqref{eq:genLyapQ} for the truncated Gramians, which take into consideration the QB nature of the problem by incorporating $\H$ and $\N$. 
\begin{proposition} \label{prop:PT}
	Let $\P_{1}$ be the solution from Proposition~\ref{prop:LinLyap}. The truncated Gramian from equation~\eqref{eq:genLyapP}  for the QB system is given as 
	\begin{equation} \label{eq:PT_QB}
		{\P}_\mathcal{T} = \P_1 + \frac{1}{2\alpha} \begin{bmatrix} \widetilde{\P}_{11}(\alpha) & \widetilde{\P}_{12}(\alpha) \\ \widetilde{\P}_{12}(\alpha)^\top & \widetilde{\P}_{22}(\alpha)  \end{bmatrix}
	\end{equation}
	via the following equations:
		\begin{align*}
			\widetilde{\P}_{22}(\alpha) & =  \H_{21}(\alpha)  (\P_{11} \otimes \P_{11}) \H_{21}^\top(\alpha) + \N_{21}\P_{11} \N_{21}^\top, \\
			\widetilde{\P}_{12}(\alpha) & = - (\A_{11} - \alpha \I )^{-1} \A_{12}\widetilde{\P}_{22}(\alpha),  \\
			\bzero & = \A_{11} \widetilde{\P}_{11}(\alpha) + \widetilde{\P}_{11}(\alpha) \A_{11}^\top + [\A_{12}\widetilde{\P}_{12}(\alpha)^\top + \widetilde{\P}_{12}(\alpha) \A_{12}^\top] .
		\end{align*}
\end{proposition}
\begin{proof}
	The generalized controllability Lyapunov equation~\eqref{eq:genLyapP} is a linear equation, therefore we can decompose the truncated Gramian $\P_\mathcal{T} = \P_1 + \P_2(\alpha)$, where $\P_1$ is the solution  from Proposition~\ref{prop:LinLyap}. Therefore, $\P_2(\alpha)$ is the solution to
	\begin{equation*}
		\A(\alpha)\P_2 + \P_2\A(\alpha)^\top + \H(\alpha)(\P_1\otimes \P_1) \H(\alpha)^\top +  \N \P_1 \N^\top  = 0.
	\end{equation*}
	With the matrices from above, we compute 
	\begin{align*}
		\N\P_1\N^\top & = \begin{bmatrix} \bzero & \bzero \\ \N_{21} & \N_{22} \end{bmatrix}  \begin{bmatrix} \P_{11} & \bzero \\ \bzero & \bzero  \end{bmatrix} \begin{bmatrix} \bzero & \N_{21}^\top \\ \bzero & \N_{22}^\top \end{bmatrix}  = \begin{bmatrix} \bzero & \bzero \\ \bzero & \N_{21}\P_{11}\N_{21}^\top  \end{bmatrix}, \\
		\H(\alpha)(\P_1\otimes \P_1) \H(\alpha)^\top & 
		=  \begin{bmatrix} \bzero & \ \bzero \\ \bzero & \  \H_{21}(\alpha)(\P_{11}\otimes \P_{11}) \H_{21}(\alpha)^\top
		\end{bmatrix}.
	\end{align*}
	The symmetric positive semi-definite solution $\P_2(\alpha)$ needs to satisfy 
	\begin{align*}
		\bzero & = 	\begin{bmatrix} \A_{11} & \A_{12} \\ \bzero & -\alpha \I \end{bmatrix} 
		\begin{bmatrix} \widehat{\P}_{11} & \widehat{\P}_{12} \\ \widehat{\P}_{12}^\top & \widehat{\P}_{22} \end{bmatrix} + 
		\begin{bmatrix} \widehat{\P}_{11} & \widehat{\P}_{12} \\ \widehat{\P}_{12}^\top & \widehat{\P}_{22} \end{bmatrix}
		\begin{bmatrix} \A_{11}^\top & \bzero \\  \A_{12}^\top  & -\alpha \I \end{bmatrix} \\
		& + \begin{bmatrix} \bzero & \bzero \\ \bzero & \H_{21}(\alpha)(\P_{11}\otimes \P_{11}) \H_{21}(\alpha)^\top \end{bmatrix}
		+ \begin{bmatrix} \bzero & \bzero \\ \bzero & \N_{21}\P_{11}\N_{21}^\top  \end{bmatrix},
	\end{align*}
	which yields the equations
	\begin{align}
		\A_{11} \widehat{\P}_{11} + \widehat{\P}_{11} \A_{11}^\top + \A_{12} \widehat{\P}_{12}^\top + \widehat{\P}_{12} \A_{12}^\top & =\bzero, \label{eq:1}\\
		(\A_{11} - \alpha \I ) \widehat{\P}_{12} + \A_{12} \widehat{\P}_{22} & =\bzero, \label{eq:2}\\
		\widehat{\P}_{12}^\top (\A_{11}^\top - \alpha \I )+ \widehat{\P}_{22}\A_{12}^\top &=\bzero, \label{eq:3} \\
		-2 \alpha \widehat{\P}_{22} + \H_{21}(\alpha)(\P_{11} \otimes \P_{11})\H_{21}(\alpha)^\top + \N_{21} \P_{11} \N_{21}^\top &=\bzero.
	\end{align}
	We define $\widetilde{\P}_{22}(\alpha) = [\H_{21}(\alpha)(\P_{11} \otimes \P_{11})\H_{21}(\alpha)^\top + \N_{21} \P_{11} \N_{21}^\top]$, so that $\widehat{\P}_{22}(\alpha) = \frac{1}{2\alpha} \widetilde{\P}_{22}(\alpha)$. 
	Due to the symmetry of $\widehat{\P}_{22}(\alpha)$, equations \eqref{eq:2} and \eqref{eq:3} are the same, and after substitution, we obtain 
$
		\widehat{\P}_{12}(\alpha) = - \frac{1}{2\alpha}(\A_{11} - \alpha \I )^{-1} \A_{12} \widetilde{\P}_{22}(\alpha) := \frac{1}{2\alpha}\widetilde{\P}_{12}(\alpha), 
$
	with $\widetilde{\P}_{12}(\alpha) : = - (\A_{11} - \alpha \I )^{-1} \A_{12} \widetilde{\P}_{22}(\alpha)$.
	Substitution of $\widehat{\P}_{12}(\alpha)$ into equation~\eqref{eq:1} yields the result. 
\end{proof}
\begin{proposition}  \label{prop:QT}
	Let the matrix $\Q_1$ be as in Proposition~\ref{prop:LinLyap}. The truncated Gramian $\Q_{\mathcal{T}}$ that solves equation~\eqref{eq:genLyapQ} is given by
	\begin{equation}  \label{eq:QT_QB}
		\Q_\mathcal{T} = \Q_1 + \begin{bmatrix}  \widehat{\Q}_{11} (\alpha)& \widehat{\Q}_{12}(\alpha) \\  \widehat{\Q}_{12}(\alpha)^\top & \widehat{\Q}_{22}(\alpha) \end{bmatrix},
	\end{equation}
	where
		\begin{align*}
			\bzero & = \A_{11}^\top \widehat{\Q}_{11}(\alpha) + \widehat{\Q}_{11}(\alpha) \A_{11} + \widetilde{\H}_{11}(\alpha) + \frac{1}{2\alpha} \N_{21}^\top \widetilde{\Q}_{22}(\alpha) \N_{21} , \\
			\widehat{\Q}_{12}(\alpha) & = - (\A_{11}^\top- \alpha \I )^{-1} \left (\widetilde{\H}_{12}(\alpha) + \frac{1}{2\alpha} \N_{21}^\top \widetilde{\Q}_{22}(\alpha) \N_{22} + \widehat{\Q}_{11}(\alpha) \A_{12} \right ), \\
			\widehat{\Q}_{22}(\alpha) & = \frac{1}{2\alpha} \left[   \A_{12}^\top \widehat{\Q}_{12}(\alpha) + \widehat{\Q}_{12}^\top(\alpha) \A_{12} + \widetilde{\H}_{22}(\alpha) + \frac{1}{2\alpha} \N_{22}^\top \widetilde{\Q}_{22}(\alpha) \N_{22} \right] .
		\end{align*}	
\end{proposition}
\begin{proof}
	We begin by computing
	\begin{equation*}
		\N^\top \Q_1 \N  = \begin{bmatrix} \bzero & \N_{21}^\top \\ \bzero & \N_{22}^\top \end{bmatrix}    \begin{bmatrix} \Q_{11} & \Q_{12}(\alpha) \\ \Q_{12}^\top(\alpha) & \frac{1}{2\alpha} \widetilde{\Q}_{22}(\alpha) \end{bmatrix} \begin{bmatrix} \bzero & \bzero \\ \N_{21} & \N_{22} \end{bmatrix}  
		= \frac{1}{2\alpha} \begin{bmatrix}
			\N_{21}^\top \widetilde{\Q}_{22} \N_{21} &  \N_{21}^\top \widetilde{\Q}_{22} \N_{22} \\ \N_{22}^\top \widetilde{\Q}_{22} \N_{21} & \N_{22}^\top \widetilde{\Q}_{22} \N_{22} \\ \end{bmatrix},
	\end{equation*}
	and for notational convenience partition the full matrix 
	\begin{equation*}
		\mathcal{H}^{(2)}(\alpha)(\P_1\otimes \Q_1)(\mathcal{H}^{(2)}(\alpha))^\top =: \begin{bmatrix}
			\widetilde{\H}_{11}(\alpha) & \widetilde{\H}_{12}(\alpha) \\ \widetilde{\H}_{12}(\alpha)^\top & \widetilde{\H}_{22}(\alpha) \\
		\end{bmatrix},
	\end{equation*}
	which is a symmetric matrix since $\P_1$ and $\Q_1$ are symmetric Gramians. 
	We then solve for the truncated Gramian via
	\begin{align*}
		\bzero & = \begin{bmatrix}\A_{11}^\top & \bzero \\  \A_{12}^\top  & -\alpha \I \end{bmatrix}
		\begin{bmatrix} \widehat{\Q}_{11} & \widehat{\Q}_{12} \\ \widehat{\Q}_{12}^\top & \widehat{\Q}_{22} \end{bmatrix} + 
		\begin{bmatrix} \widehat{\Q}_{11} & \widehat{\Q}_{12} \\ \widehat{\Q}_{12}^\top & \widehat{\Q}_{22} \end{bmatrix}
		\begin{bmatrix} \A_{11} & \A_{12} \\ \bzero & -\alpha \I \end{bmatrix} \nonumber \\
		& + \begin{bmatrix} \widetilde{\H}_{11}(\alpha) & \widetilde{\H}_{12}(\alpha) \\ \widetilde{\H}_{12}(\alpha)^\top & \widetilde{\H}_{22}(\alpha) \\ \end{bmatrix}
		+ \frac{1}{2\alpha} \begin{bmatrix}
			\N_{21}^\top \widetilde{\Q}_{22} \N_{21} &  \N_{21}^\top \widetilde{\Q}_{22} \N_{22} \\ \N_{22}^\top \widetilde{\Q}_{22} \N_{21} & \N_{22}^\top \widetilde{\Q}_{22} \N_{22} \\ \end{bmatrix}.
	\end{align*}
	Thus, the Gramian solution has to satisfy the equations
	\begin{align*}
		\bzero & = \A_{11}^\top \widehat{\Q}_{11} + \widehat{\Q}_{11} \A_{11} + \widetilde{\H}_{11}(\alpha) + \frac{1}{2\alpha} \N_{21}^\top \widetilde{\Q}_{22} \N_{21} , \\
		\bzero & =  (\A_{11}^\top-\alpha \I) \widehat{\Q}_{12} + \widehat{\Q}_{11} \A_{12} + \widetilde{\H}_{12}(\alpha) + \frac{1}{2\alpha} \N_{21}^\top \widetilde{\Q}_{22} \N_{22}, \\
		\bzero & = \A_{12}^\top \widehat{\Q}_{12} - 2 \alpha \widehat{\Q}_{22}  + \widehat{\Q}_{12}^\top \A_{12} + \widetilde{\H}_{22}(\alpha) + \frac{1}{2\alpha} \N_{22}^\top \widetilde{\Q}_{22} \N_{22},
	\end{align*}
	from which we get the stated result.
\end{proof}	
By decomposing the Lyapunov solution into block form, we can again compute the truncated Gramians by only working in the original dimensions (not the lifted dimensions). Due to the dependence on $1/\alpha$  the Gramians $\Q_{\mathcal{T}}, \P_{\mathcal{T}}$ are not defined for $\alpha \rightarrow 0$ and taking the limit $\alpha \rightarrow 0$ can therefore result in numerical difficulties.
 We summarize our proposed approach below.
\begin{proposition} \label{prop:BalancingAlg}
	An approximate balancing transformation for the QB structure arising in lifting of nonlinear systems  can be computed as follows:
	\begin{enumerate}
		\item Choose  $\alpha \in (0,\infty)$  and compute $\P_\mathcal{T}, \Q_\mathcal{T}$ from Propositions~\ref{prop:PT} and \ref{prop:QT} with $\A(\alpha)$.
		\item Compute Choleksy factors $\ \P_\mathcal{T} = \L_{\P_\mathcal{T}}\L_{\P_\mathcal{T}}^\top$ and $\ \Q_\mathcal{T} = \L_{\Q_\mathcal{T}} \L_{\Q_\mathcal{T}}^\top$.
		\item Compute SVD $\ \L_{\Q_\mathcal{T}}^\top \L_{\P_\mathcal{T}} = \mathcal{U} \mathcal{S} \mathcal{V}^\top$; denote rank $r$ approximation as $\ \mathcal{U}_r \mathcal{S}_r \mathcal{V}_r^\top$.
		\item Projection matrices $\ \W = \L_{\Q_\mathcal{T}} \mathcal{U}_r \mathcal{S}_r^{-1/2}$ and $\ \V = \L_{\P_\mathcal{T}} \mathcal{V}_r \mathcal{S}_r^{-1/2}$.
		\item Matrices for QB-ROM~\eqref{eq:ROMQB1}--\eqref{eq:ROMQB2} are
		\begin{equation*}
			\widehat{\A} = \W^\top \A \V, \quad \widehat{\B} = \W^\top \B, \quad \widehat{\C} = \C \V, \quad \widehat{\N}_k = \W^\top \N_k \V, \quad \widehat{\H} = \W^\top \H (\V\otimes \V).
		\end{equation*}
	\end{enumerate}	
\end{proposition}

\section{Numerical Results} \label{sec:numerics}
For the tubular reactor model in Section~\ref{sec:tubReactor_model}, the parameters are $\Dam=0.17$, $Pe=25$, $\mathcal{B}=0.5$, $\beta=2.5$, $\theta_{\text{ref}}\equiv 1$,  $\gamma=5$, and $n=199$. We compare the approximate balanced truncation ROM, denoted QB-BT, from Proposition~\ref{prop:BalancingAlg} with a  ROM computed from POD with discrete empirical interpolation~\cite{deim2010}, denoted as POD-DEIM. For the predictive comparison, the full-order model and ROMs are simulated until $t_f = 30$s with different inputs  $u(t)$ and initial conditions $\x(0)=\x_0$ as given below. 

The POD-DEIM models are obtained as follows: We simulate the full-order model until $t_\text{train}=15$s with training input $u_\text{train}(t)$ and training initial condition $\x_{0,\text{train}}$. We record a snapshot every $0.01$s and store the snapshots in a matrix $\X$, compute the POD basis $\V \in \mathbb{R}^{n\times r}$ of dimension $r$, and project the system matrices in~\eqref{eq:tub_FOM1}--\eqref{eq:tub_FOM2} onto $\V$. For every $r$-dimensional POD-DEIM ROM we use $r$ DEIM interpolation points using the QDEIM algorithm from~\cite{drmac2016QDEIM}.

For the QB-BT models, we discussed in Section~\ref{subsec:stabilization} that the parameter $\alpha\in (0,\infty)$. Here, we found that  $\alpha=20$ resulted in the most accurate balanced ROMs and that for $\alpha \ll 1$ the Gramians' subspaces would lead to unstable ROMs. This could be related to the influence of the field of values of $\A(\alpha)$ on the decay of the singular values of the solution to the Lyapunov equation~\cite{baker2015fast}. Here, for $\alpha>\alpha_c=2.6$, the field of values of $\A(\alpha)$ is in the open left-half plane, which we can choose as a selection criterion for $\alpha$.

Reduced-order models constructed via POD-DEIM depend on the snapshot data that is obtained from training simulations. In contrast, balancing transformations only depend on the matrices defining the full-order QB system $\A,\H, \N, \B, \C$; the basis generation does not require a choice of training data. To illustrate this point, we consider four test cases with different testing conditions and for which the POD-DEIM training choices vary.  For each test case we compare the output error $\sum_{i=1}^s  \vert y(t_i) - y_r(t_i)\vert$ with $s=3,000$ for the POD-DEIM and QB-BT ROMs.  We detail each case  below. \\

\textbf{Case 1}: $u(t) = \cos(t), \ u_\text{train}(t) = u(t), \x_{0,\text{train}} = \x_0$. \ For this test case, training and testing conditions are the same, so we expect POD-DEIM ROMs to accurately reproduce the full-order model (FOM) simulations. The output quantity of interest of the FOM compared to two ROMs with $r=20$ basis functions for this test case is shown in Figure~\ref{fig:test12}, left. Table~\ref{tab:test1} shows the output error for various ROM model sizes. We see that the POD-DEIM ROM outperforms the QB-BT ROM in this case. 

\begin{table}
	\centering
	\caption{Case~1: Output error $\sum_{i=1}^s  \vert y(t_i) - y_r(t_i)\vert$ in ROMs.}
	\begin{tabular}{c| ccccccccc}
		\hline
		ROM& $r$=4 & $r$=6 & $r$=8 & $r$=10 & $r$=12 & $r$=14 & $r$=16 & $r$=18 & $r$=20 \\
		\hline\hline
		QB-BT& 6.51E-04 & 7.00E-04 & 6.98E-04 & 7.09E-04 & 7.06E-04 & 6.84E-04 & 6.66E-04 & 7.05E-04 & 6.95E-04 \\
		POD-DEIM& 2.53E-05 & 1.17E-05 & 7.39E-06 & 6.93E-06 & 5.61E-06 & 5.62E-06 & 5.62E-06 & 5.63E-06 & 5.62E-06 \\
	\end{tabular}
	\label{tab:test1}
\end{table}

\textbf{Case 2}: $u(t) = \cos(t), \ u_\text{train}(t) = u(t), \ \x_{0,\text{train}} = \x_0$; noise = $10\%$
Here we use the same training input and training initial conditions as in Case~1. However, we add noise to the training data, reflecting a situation that practitioners often face where measurements are noise-corrupted. In particular, we use $\tilde{X} = X + 0.1 X(-1 + 2\Xi)$ where $\Xi = \mathcal{N}(0,1)$ is a normally distributed random variable.  Figure~\ref{fig:test12}, right, shows the model output for the FOM and both ROMs for $r=20$ where we see that the POD-DEIM model does not capture the output amplitude well. Table~\ref{tab:test2} shows the output errors for increasing ROM sizes. We observe that the QB-BT ROM outperforms the POD-DEIM model in this case, as the POD-DEIM model suffers from a two orders of magnitude loss in accuracy (compared to Table~\ref{tab:test1}) due to noisy snapshots. 
\begin{table}
	\centering
	\caption{Case~2: Output error $\sum_{i=1}^s  \vert y(t_i) - y_r(t_i)\vert$ in ROMs.}
	\begin{tabular}{c| ccccccccc}
		\hline
		ROM&  $r$=4 & $r$=6 & $r$=8 & $r$=10 & $r$=12 & $r$=14 & $r$=16 & $r$=18 & $r$=20 \\
		\hline\hline
		QB-BT& 6.51E-04 & 7.00E-04 & 6.98E-04 & 7.09E-04 & 7.06E-04 & 6.84E-04 & 6.66E-04 & 7.05E-04 & 6.95E-04 \\
		POD-DEIM& 1.47E-03 & 1.41E-03 & 1.40E-03 & 1.41E-03 & 1.39E-03 & 1.40E-03 & 1.40E-03 & 1.40E-03 & 1.40E-03 \\	
	\end{tabular}
	\label{tab:test2}
\end{table}
\begin{figure}[H]
	\centering
	\includegraphics[width=0.495\textwidth]{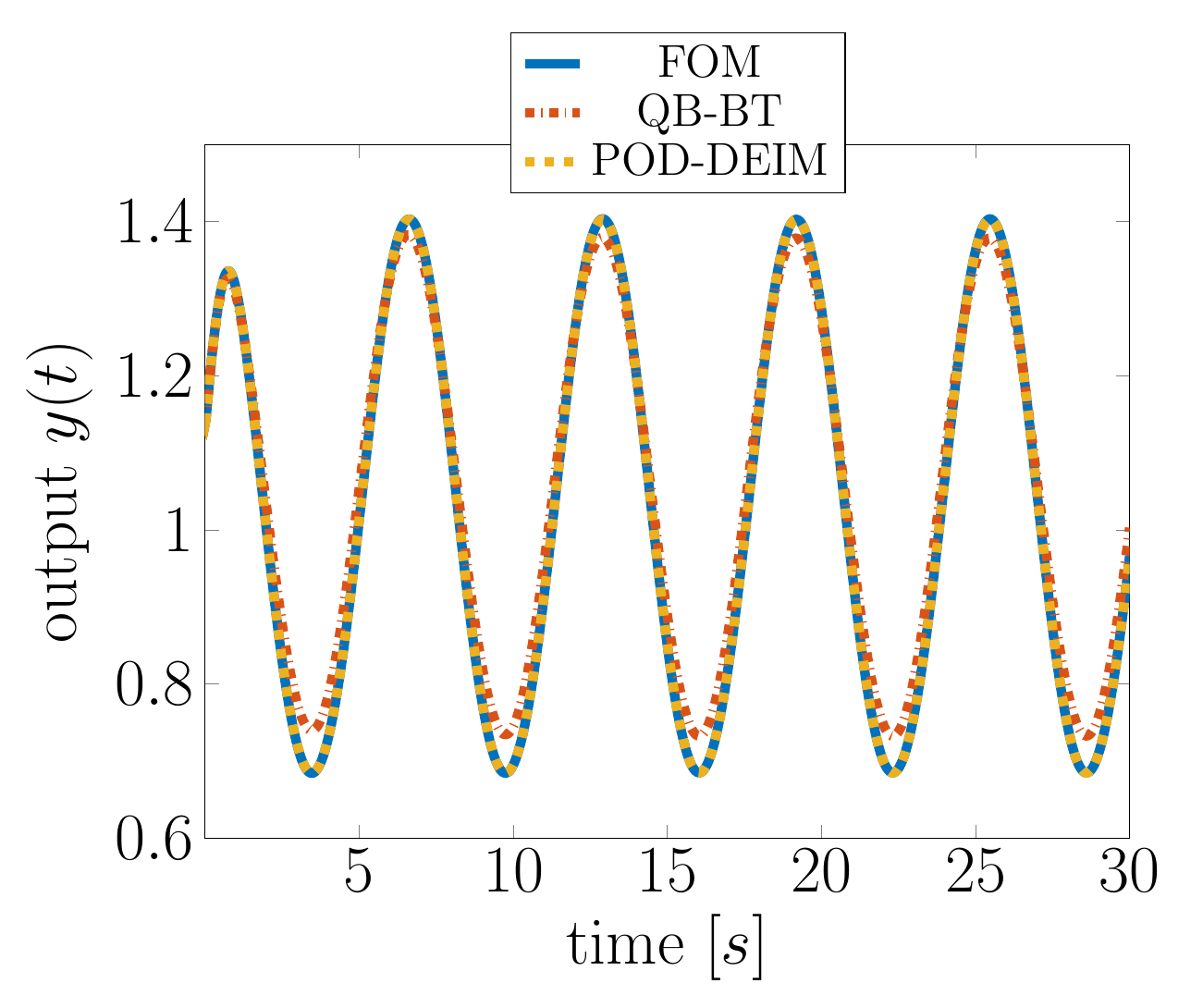}
	\includegraphics[width=0.495\textwidth]{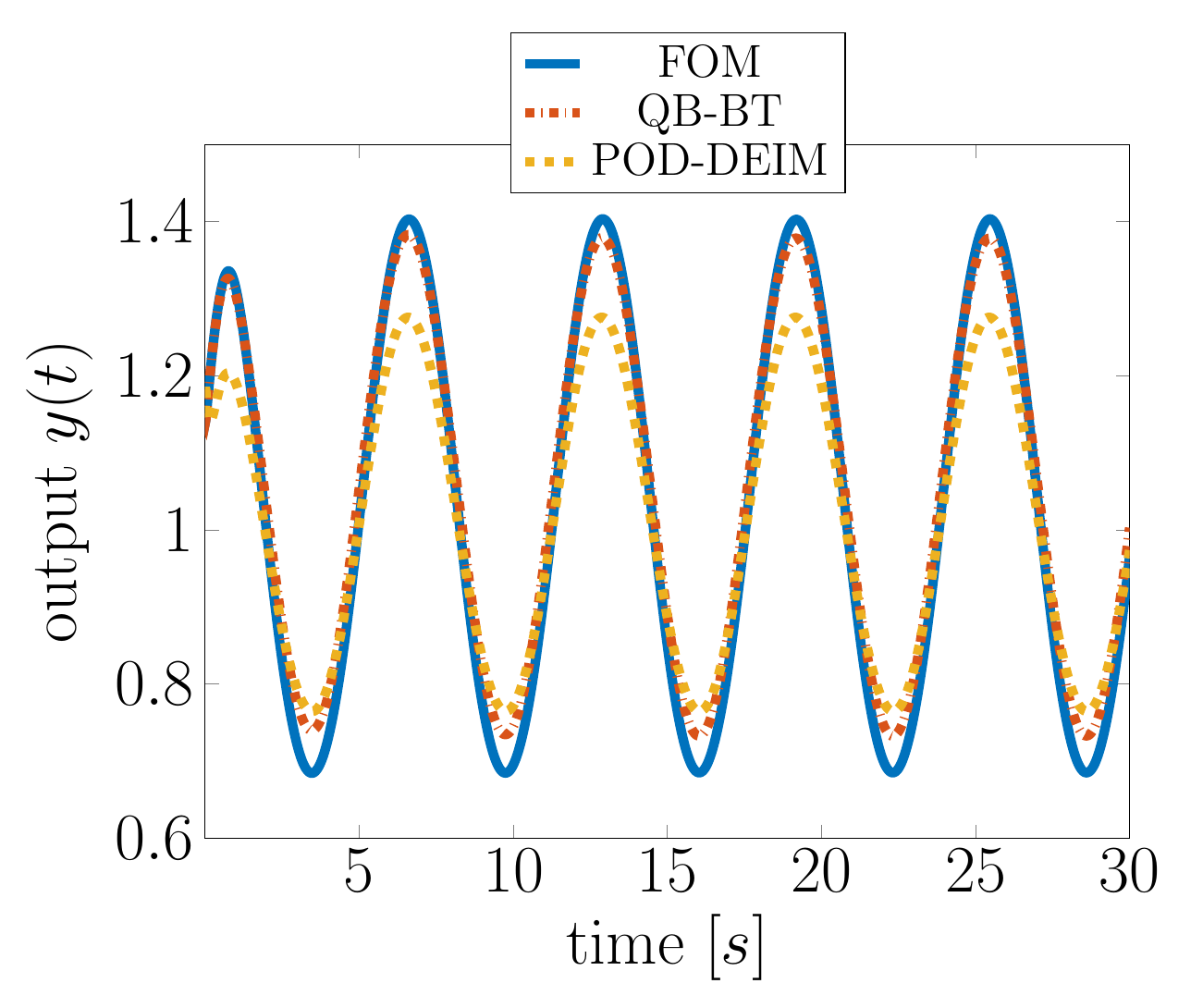}
	\caption{Model outputs for test Case 1 (left) and Case 2 (right): Comparison of FOM with two ROMs for $r=20$ basis functions, respectively.}
	\label{fig:test12}
\end{figure}

\textbf{Case 3}: $u(t) = 0.5(1+ t^2\exp(-t/4)\sin(6t)), \ u_\text{train}(t) = 0.5, \ \x_{0,\text{train}} = \x_0$
This case illustrates a scenario where a practitioner trained a model assuming that input conditions are operating at a constant equilibrium. However during operation of the plant the inputs indeed vary but the oscillations are damped and revert to the equilibrium position of $0.5$. Figure~\ref{fig:test34}, left, shows the model output for the FOM and both ROMs for $r=20$; both ROMs are reproducing the output well. Table~\ref{tab:test3} shows the output errors for increasing ROM sizes. The QB-BT ROM is again more accurate, yet at one model order, $r=20$, the POD-DEIM model is more accurate. 
\begin{table}
	\centering
	\caption{Case~3: Output error $\sum_{i=1}^s  \vert y(t_i) - y_r(t_i)\vert$ in ROMs.}
	\begin{tabular}{c| ccccccccc}
		\hline
		ROM& $r$=4 & $r$=6 & $r$=8 & $r$=10 & $r$=12 & $r$=14 & $r$=16 & $r$=18 & $r$=20 \\
		\hline\hline
		QB-BT& 5.30E-04 & 6.02E-04 & 5.41E-04 & 4.96E-04 & 5.03E-04 & 7.18E-04 & 6.48E-04 & 5.63E-04 & 5.37E-04 \\
		POD-DEIM& 2.42E-03 & 2.22E-03 & 2.15E-03 & 2.15E-03 & 1.96E-03 & 9.11E-04 & 6.55E-04 & 6.16E-04 & 2.37E-04 \\	
	\end{tabular}
	\label{tab:test3}
\end{table}

\textbf{Case 4}: $u(t) = \cos(t), \ u_\text{train}(t) = 0.5, \ \x_{0,\text{train}}(1:n) =0, \ \x_{0,\text{train}}(n+1:2n)=1$. Here, the training initial condition is different from the initial condition used for prediction, i.e., $\x_{0,\text{train}} \neq \x_{0}$. Figure~\ref{fig:test34}, right, shows the model output for the FOM and both ROMs for $r=20$; the QB-BT model slightly outperforms the POD-DEIM model, but both fall short of predicting the full amplitude. Table~\ref{tab:test4} shows the output errors for increasing ROM sizes, where the QB-BT model again seems to outperform the POD-DEIM model for all basis sizes.
\begin{table}
	\centering
	\caption{Case~4: Output error $\sum_{i=1}^s  \vert y(t_i) - y_r(t_i)\vert$ in ROMs.}
	\begin{tabular}{c| ccccccccc}
		\hline
		ROM & $r$=4 & $r$=6 & $r$=8 & $r$=10 & $r$=12 & $r$=14 & $r$=16 & $r$=18 & $r$=20 \\
		\hline\hline
		QB-BT& 3.46E-04 & 3.47E-04 & 3.77E-04 & 3.76E-04 & 3.75E-04 & 3.71E-04 & 3.41E-04 & 3.76E-04 & 3.66E-04 \\
		POD-DEIM& 8.86E-04 & 9.23E-04 & 9.26E-04 & 9.25E-04 & 9.29E-04 & 7.05E-04 & 4.51E-04 & 4.51E-04 & 4.38E-04 \\	
	\end{tabular}
	\label{tab:test4}
\end{table}
\begin{figure}[H]
	\centering
	\includegraphics[width=0.495\textwidth]{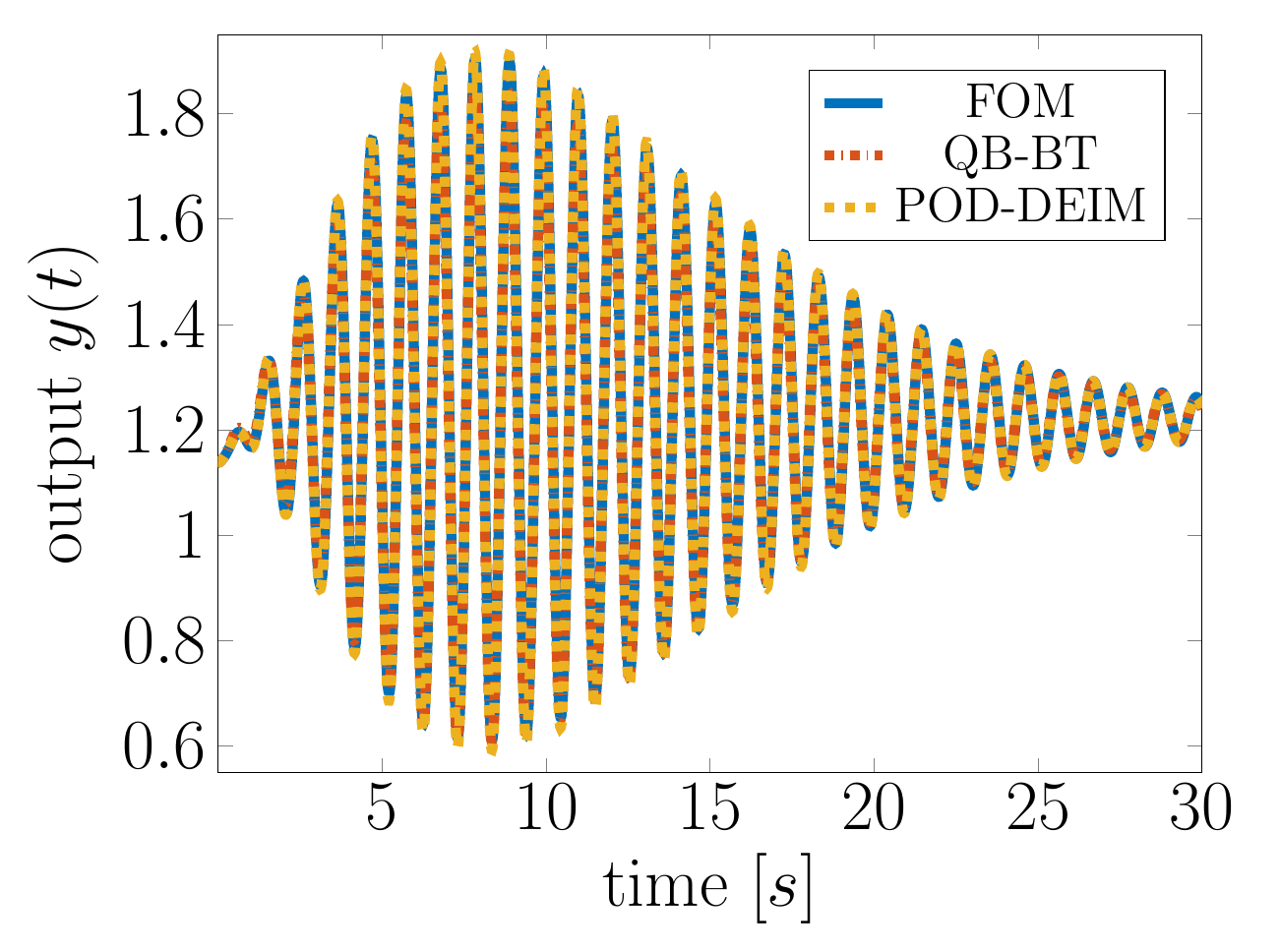}
	\includegraphics[width=0.495\textwidth]{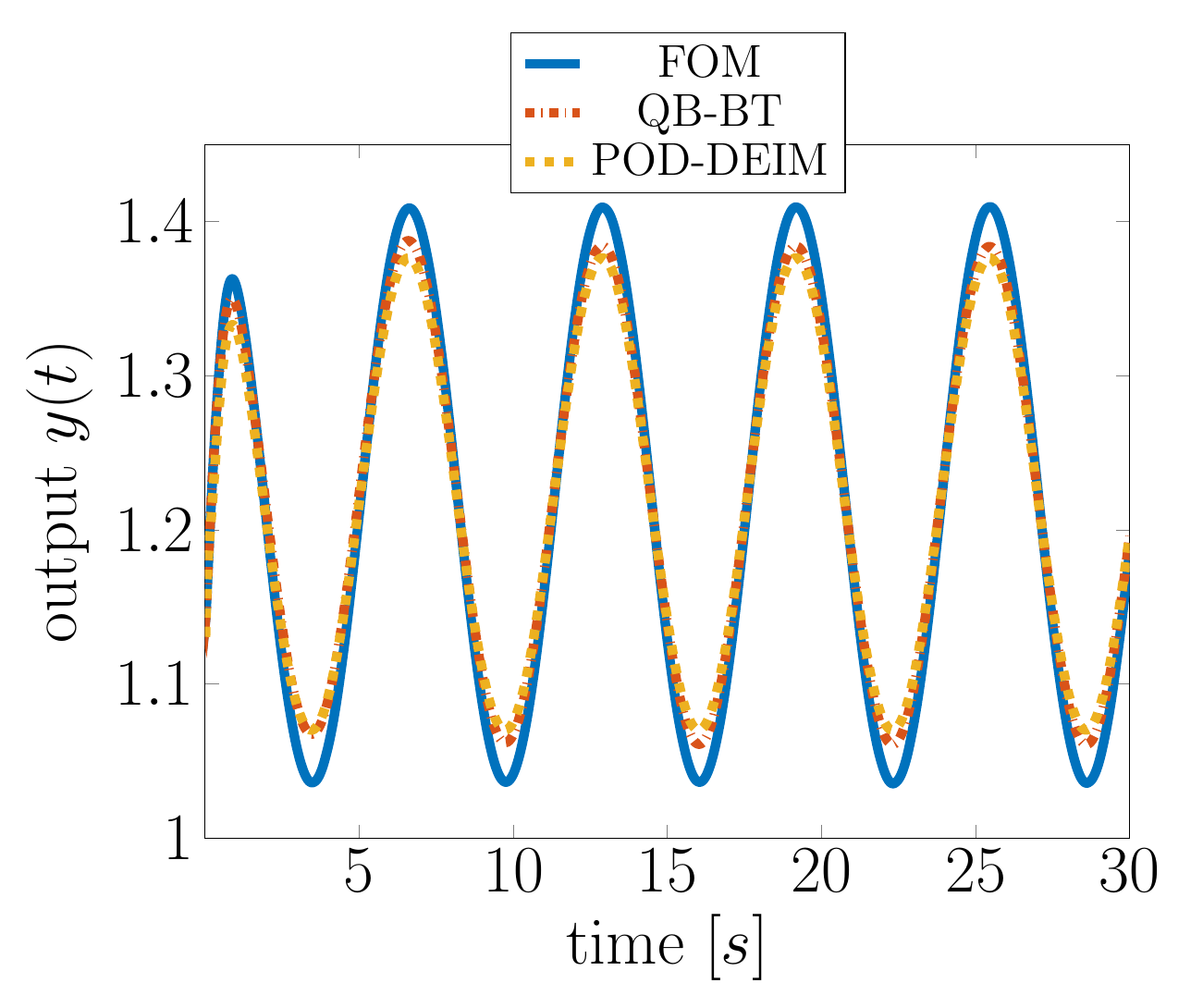}
	\caption{Model outputs for Case 3 (left) and Case 4 (right): Comparison of FOM with two ROMs for $r=20$ basis functions, respectively.}
	\label{fig:test34}
\end{figure}

\section{Conclusions}
We presented a balanced truncation model reduction approach that is applicable to a large class of nonlinear systems with time-varying and uncertain inputs. As a first step, our approach lifts the nonlinear system to quadratic-bilinear form via the introduction of auxiliary variables. As lifted systems often have system matrices with zero eigenvalues, we first introduced an artificial stabilization parameter and then derived a balancing algorithm for those lifted quadratic-bilinear systems that only requires expensive matrix computations in the original---and not in the lifted---dimension. The method was illustrated by the model reduction problem for a tubular reactor model, for which we derived the multi-stage lifting transformations and performed balanced model reduction. The numerical results showed that through our proposed method, we can obtain ROMs for nonlinear systems that are superior to POD-DEIM models in situations where a good choice of training data is not feasible.

\begin{acknowledgement}
This work has been supported in part by the Air Force Center of Excellence on Multi-Fidelity Modeling of Rocket Combustor Dynamics under award FA9550-17-1-0195, by the Air Force Office of Scientific Research (AFOSR) MURI on managing multiple information sources of multi-physics systems, Award Numbers FA9550-15-1-0038 and FA9550-18-1-0023, and the AEOLUS MMICC center under the Department of Energy grant DE-SC0019303.
The authors thank the reviewer for the very valuable comments which helped us improve our previous stabilization approach in Section~4.1. The authors thank Dr. Pawan Goyal for valuable discussions about the truncated Gramian framework. The authors thank Prof. Mark Embree for insightful feedback on the manuscript, and pointing out the connection to the field of values and the choice of $\alpha$ in our algorithm.
\end{acknowledgement}

\bibliographystyle{abbrv}
\bibliography{NLBT}

\end{document}